\documentclass[12pt, reqno]{amsart}
\usepackage{amsmath, amsthm, amscd, amsfonts, amssymb, graphicx, color, mathrsfs}
\usepackage[bookmarksnumbered, colorlinks, plainpages]{hyperref}

\newtheorem{theorem}{Theorem}[section]
\newtheorem{lemma}[theorem]{Lemma}
\newtheorem{proposition}[theorem]{Proposition}
\newtheorem{corollary}[theorem]{Corollary}
\theoremstyle{definition}
\newtheorem{definition}[theorem]{Definition}

\theoremstyle{remark}
\newtheorem{remark}[theorem]{Remark}
\numberwithin{equation}{section}

\begin{document}
\setcounter{page}{1}

\title[Boundedness of FSOs ]{$L^p$-bounds for  Fourier integral operators on the torus}

\author[D. Cardona]{Duv\'an Cardona}
\address{
  Duv\'an Cardona:
  \endgraf
  Department of Mathematics: Analysis, Logic and Discrete Mathematics 
  \endgraf
  Ghent University
  \endgraf
  Ghent
  \endgraf
  Belgium
  \endgraf
  {\it E-mail address} {\rm duvanc306@gmail.com}
  }
\author[R. Messiouene]{ Rekia Messiouene}
\address{
  Rekia Messiouene:
  \endgraf
  Laboratoire de Math\'ematiques Fondamentales et Appliqu\'ees d'Oran (LMFAO). 
  \endgraf
 Universit\'e Oran1. 
  \endgraf
   B.P. 1524 El M'naouar,
  \endgraf
  Oran
  \endgraf
  Algeria
  \endgraf
  {\it E-mail address} {\rm rekiamessiouene@yahoo.fr}
  }
\author[A. Senoussaoui]{  Abderrahmane Senoussaoui}
\address{
  Abderrahmane Senoussaoui:
  \endgraf
  Laboratoire de Math\'ematiques Fondamentales et Appliqu\'ees d'Oran (LMFAO). 
  \endgraf
 Universit\'e Oran1. 
  \endgraf
   B.P. 1524 El M'naouar,
  \endgraf
  Oran
  \endgraf
  Algeria
  \endgraf
  {\it E-mail address} {\rm  senoussaoui$\_$abdou@yahoo.fr }
  }

\dedicatory{Dedicated to the $47^{th}$ birthday of  Michael Ruzhansky}



\begin{abstract} In this paper we investigate the mapping properties of periodic Fourier integral operators in $L^p(\mathbb{T}^n)$-spaces. The operators considered are associated to periodic symbols (with limited regularity) in the sense of Ruzhansky and Turunen. 

\end{abstract} \maketitle

\tableofcontents
\section{Introduction}
In this paper we investigate the $L^p$-boundedeness  of periodic Fourier integral operators (also called Fourier series operators). Let us consider the $n$-dimensional torus, $\mathbb{T}^n:=\mathbb{R}^n/\mathbb{Z}^n ,$ and let us  choose a suitable function $a:\mathbb{T}^n\times \mathbb{Z}^n\rightarrow \mathbb{C}.$ Then, the periodic Fourier integral operator associated to the so-called symbol $a$  is formally defined by the   series
\begin{equation}
Af(x):=\sum_{\xi\in\mathbb{Z}^n}e^{2\pi i\phi(x,\xi)}a(x,\xi)(\mathscr{F}_{{\mathbb{T}^n}}f)(\xi),\,\,f\in C^{\infty}(\mathbb{T}^n).\footnote{ For $G=\mathbb{T}^n$ or $G=\mathbb{R}^n,$  we denote by $\mathscr{F}_{G}f$ the Fourier transform of $f\in L^1(G),$ defined by
$
    \mathscr{F}_Gf(\xi)=\int_{G}e^{-i2\pi x\cdot \xi}f(x)dx,\,\,\xi\in \widehat{G},
$ where $\widehat{\mathbb{T}}^n=\mathbb{Z}^n$ and  $\widehat{\mathbb{R}}^n=\mathbb{R}^n$.}
\end{equation}
These operators where introduced by M. Ruzhansky and V. Turunen in \cite[Chapter 4]{Ruz} and they appear in solutions of  hyperbolic differential equations with periodic conditions (see \cite[Pag. 410]{Ruz}). In this paper we give conditions on the symbol $a$  and on the phase function $\phi$ are so that the operator $A$ extends to a bounded operator on $L^p(\mathbb{T}^n).$ 

Periodic Fourier integral operators are analogues on the torus of Fourier integral operators (FIOs) on $\mathbb{R}^n,$ which have the form
\begin{equation}
T_{a,\phi}f(x):=\int_{\mathbb{R}^n}e^{2\pi i\phi(x,\xi)}a(x,\xi)(\mathscr{F}_{\mathbb{R}^n }f)(\xi)d\xi,
\end{equation}
where $\mathscr{F}_{\mathbb{R}^n }f$ is the Fourier transform of  $f,$ or more generally of FIO formally defined by
\begin{equation}
T_{a,\phi}f(x)=\int_{\mathbb{R}^{2n}}e^{2\pi i\phi(x,\xi)-2\pi i y\cdot\xi}a(x,y,\xi)f(y)dyd\xi.
\end{equation}
FIOs are used to 
express solutions to Cauchy problems of hyperbolic equations and   transform operators or equations to other  simpler ones according to Egorov's theorem (see H\"ormander \cite{Hor2}). The problem of finding mapping properties of FIOs on $L^p$-spaces have been extensively investigated. The case where the phase function is given by $\phi(x,\xi)=x\cdot \xi$ and the symbol  $a(x,y,\xi)=a(x,\xi)$ is considered in the variables $(x,\xi)$ reduces the problem to pseudo-differential operators and Fourier multipliers \cite{Hor1,Hor2,JKohnLNirenberg}. In this case  symbol inequalities of the type
\begin{equation}
|\partial_x^\beta\partial^\alpha_\xi a(x,\xi)|\leq C_{\alpha,\beta}(1+|\xi|)^{-m_p-\rho|\alpha|+\delta|\beta|},\,\,m_p=n(1-\rho)\left|\frac{1}{p}-\frac{1}{2}\right|,\,0\leq \delta<\rho\leq 1,
\end{equation}
are sufficient conditions for the $L^p$-boundedness. The historical development of the boundedness properties for pseudo-differential operators on $L^p$-spaces can be found in Wang \cite{LW}. 

In the case of general phases, according to the theory of FIOs developed by H\"ormander \cite{Hor71}, the phase functions $\phi$  are positively homogeneous
of order 1  and smooth at $\xi\neq 0,$  and the  symbols satisfy estimates of the form
\begin{equation}
\sup_{(x,y)\in K}|\partial_x^\beta\partial^\alpha_\xi a(x,y,\xi)|\leq C_{\alpha,\beta,K}(1+|\xi|)^{\kappa-|\alpha|}
\end{equation}
for every compact subset $K$ of $\mathbb{R}^{2n}.$ So, as it was pointed out in Ruzhansky and Wirth \cite{RuzSugi}, $L^p$-properties of FIO can be summarized as follows.
\begin{itemize}
\item If $\kappa\leq 0,$ then $T$ is $(L^2_{\textnormal{comp}},L^2_{\textnormal{loc}})$-bounded (H\"ormander\cite{Hor71} and Eskin\cite{Eskin}).

\item If $\kappa\leq\kappa_p:= -(n-1)\left|\frac{1}{p}-\frac{1}{2}\right|,$ then $T$ is $(L^p_{\textnormal{comp}},L^p_{\textnormal{loc}})$-bounded (Seeger, Sogge and Stein\cite{SSS91}).

\item If $\kappa\leq - \frac{1}{2}(n-1),$ then $T$ is $(H^1_{\textnormal{comp}},L^1_{\textnormal{loc}})$-bounded (Seeger, Sogge and Stein\cite{SSS91}).

\item If $\kappa\leq- \frac{1}{2}(n-1),$ then $T$ is locally weak $(1,1)$ type (Terence Tao\cite{Tao}).

\item Other conditions can be found in Miyachi \cite{Miyachi}, Peral\cite{Peral}, Asada and Fujiwara\cite{AF}, Fujiwara\cite{Fuji}, Kumano-go\cite{Kumano-go}, Coriasco and Ruzhansky \cite{CoRu1,CoRu2}, Ruzhansky and Sugimoto\cite{RuzSugi01,RuzSugi02,RuzSugi03,RuzSugi}, Beltran, Hickman and Sogge \cite{BeltranHickmanSogge2018} and Ruzhansky \cite{M. Ruzhansky}. The boundedness and the compactness on $L^2$ for a class of these operators but with a general class of symbols can be found in \cite{SENOUS1} and \cite{SENOUS2}. However the boundedness on $L^2 $ and $L^p $ of these classes of operators but with semi classical parameter can be found on \cite{Elo} and \cite{Har}.

\item A periodic version for the $L^2$-result by H\"ormander and Eskin mentioned above, was proved by Ruzhansky and Turunen (see Theorem \ref{TeoremaRT}). See also dispersive estimates for FIOs in Ruzhansky and Wirth \cite{RWirth}.
\end{itemize}

Our results are proved in the framework of   periodic  operators on the torus $\mathbb{T}^n.$  The subject was developed by several authors over decades (see Agranovich\cite{ag}, McLean\cite{Mc}, Turunen and Vainikko\cite{tur}, Ruzhansky and Turunen\cite{Ruz-2}) and generalized to arbitrary compact Lie groups in the fundamental book \cite{Ruz} by Ruzhansky and Turunen. The first step in the analysis on the torus is the definition of periodic pseudo-differential operators, which are linear operators of the form
\begin{equation}
a(x,D)f(x):=\sum_{\xi\in\mathbb{Z}^n}e^{i2\pi x\cdot \xi}a(x,\xi)(\mathscr{F}_{{\mathbb{T}^n}}f)(\xi).
\end{equation}
The calculus of pseudo-differential operators has been treated, e.g., in \cite{Ruz-2,Ruz} and its mapping properties on $L^p(\mathbb{T}^n)$-spaces can be found in the works of the first author \cite{Duvan2,Duvan3,Duvan4,Duvan5}, in Delgado\cite{Profe}, Molahajloo and Wong\cite{s1,s2,m} and Cardona and Kumar \cite{KumarCardona,KumarCardona2}. It is important to mention that  $L^p$-estimates for pseudo-differential operators on compact Lie groups (e.g., the $n$-dimensional torus $\mathbb{T}^n$ or the compact Lie groups $\textrm{SU}(2),\textnormal{SO}(3),$ etc.) can be found in Delgado and Ruzhansky \cite{DR4}. 
The main results in this paper are Theorems \ref{theorem 1.1 } and \ref{Teorema1Intro} below. In Theorem \ref{theorem 1.1 } we investigate how the $L^p$-boundedness of FIOs implies the $L^p$-boundedness of FSOs.
\begin{theorem}\label{theorem 1.1 }
Let $1< p<\infty.$ Let us assume that $\phi$ is a real valued continuous function defined on $\mathbb{T}^n\times\mathbb{R}^n.$ If  $a:\mathbb{T}^n\times\mathbb{R}^n\rightarrow \mathbb{C}$ is a continuous bounded function and the Fourier integral operator
\begin{equation}
T_{\phi,a}f(x):=\int_{\mathbb{R}^n}e^{2\pi i\phi(x,\xi)}a(x,\xi)(\mathscr{F}_{\mathbb{R}^n}f)(\xi)d\xi
\end{equation} extends to a bounded operator  $T_{\phi,a}:L^p(\mathbb{R}^n)\rightarrow L^p(\mathbb{R}^n),$ then the periodic Fourier integral operator
\begin{equation}\label{cp}
A_{\phi,a}f(x):=\sum_{\xi\in\mathbb{Z}^n}e^{2\pi i\phi(x,\xi)}a(x,\xi)(\mathscr{F}_{{\mathbb{T}^n}}f)(\xi)
\end{equation} also extends to a bounded operator $A_{\phi,a}:L^p(\mathbb{T}^n)\rightarrow L^p(\mathbb{T}^n).$ Moreover, for  some  $C_p>0,$ the estimate $\Vert A_{\phi,a} \Vert_{\mathscr{B}(L^p(\mathbb{T}^n))}\leq C_p\Vert T_{\phi,a}\Vert_{\mathscr{B}(L^p(\mathbb{R}^n))}$ holds true.
\end{theorem}
In the next result we establish the boundedness of FSOs.  Recall that the operator $\Delta_\xi$ is the usual difference operator acting on sequences, see Definition \ref{diferenceoperator}.
\begin{theorem}\label{Teorema1Intro} 
Let us assume that $\phi:\mathbb{T}^n\times \mathbb{R}^n\rightarrow\mathbb{R}$ is a   real-valued phase function positively homogeneous of order $1$ in $\xi\neq 0$. 
 Let us assume that $\partial_{x}^{\gamma'}\partial_{\xi}^\gamma\phi\in S^0_{0,0}(\mathbb{T}^n\times (\mathbb{R}^n\setminus \{0\}))$ when $|\gamma|=|\gamma'|=1,$ that
\begin{equation}
    |\textnormal{det}(\partial_{y}\partial_{\xi}\phi(y,\xi))|\geq C>0,\,\,\,|\partial^\alpha_{y}\phi(y,\xi)|\leq C_{\alpha}|\xi|,\,\,\xi\neq 0,
\end{equation}
\begin{equation}
    \langle \nabla_{\xi}\phi(y,\xi)\rangle\asymp  1,\,\,\,\langle \nabla_{y}\phi(y,\xi)\rangle\asymp \langle \xi\rangle, 
\end{equation} 
and the symbol inequalities
\begin{equation}
|\partial_x^\beta\Delta_\xi^\alpha a(x,\xi)|\leq C_{\alpha,\beta}\langle \xi\rangle^{\kappa-|\alpha|},\,\,\kappa\leq\kappa_p:=-(n-1)\left|\frac{1}{p}-\frac{1}{2}\right|,\,\,|\beta|\leq \left[\frac{n}{p}\right]+1,
\end{equation} hold true for every $(x,\xi)\in \mathbb{T}^n\times \mathbb{Z}^n$. Then, the periodic Fourier integral operator in \eqref{cp} extends to a bounded operator  $A_{\phi,a}:L^p(\mathbb{T}^n)\rightarrow L^p(\mathbb{T}^n)$  for all  $1<p<\infty.$
\end{theorem}
Theorem  \ref{Teorema1Intro} gives the $L^2$-boundedness of FSOs provided that $\kappa\leq 0.$ Our conditions however are different from the following sharp $L^2$-result due to  Ruzhansky and  Turunen.
\begin{theorem}[Ruzhansky-Turunen]\label{TeoremaRT} 
Let us assume that $\phi:\mathbb{T}^n\times \mathbb{Z}^n\rightarrow\mathbb{R}$ is a   real-valued phase function,  homogeneous of order $1$ in $\xi\neq 0$. Let us assume that $\Delta^\gamma_\xi\phi\in S^0_{0,0,2n+1,0}(\mathbb{T}^n\times \mathbb{R}^n)$ when $|\gamma|=1$ and the symbol  inequalities
\begin{equation}
|\partial_x^\beta a(x,\xi)|\leq C,\,\,\,|\beta|\leq 2n+1,
\end{equation} hold true. Assume also that
\begin{equation}
|\nabla_x\phi(x,\xi)-\nabla_x\phi(x,\xi')|\geq C|\xi-\xi'|,\,\xi,\xi'\in\mathbb{Z}^n.
\end{equation}
Then, the periodic Fourier integral operator in \eqref{cp} extends to a bounded operator $A_{\phi,a}:L^2(\mathbb{T}^n)\rightarrow L^2(\mathbb{T}^n)$.
\end{theorem}

In relation with the $L^2$-results mentioned above, we present the following dispersive estimate for a family of periodic Fourier integral operators of the form
\begin{equation}\label{cptt}
A_{t}f(x):=\sum_{\xi\in\mathbb{Z}^n}e^{i2\pi x\xi+2\pi it\phi(t,x,\xi)}a(t,x,\xi)(\mathscr{F}_{{\mathbb{T}^n}}f)(\xi),\,0<t_0\leq t<\infty.
\end{equation} The corresponding assertion is the following.
\begin{theorem} Let us consider the parametrized family of periodic Fourier integral operators 
\begin{equation}
A_{t}f(x):=\sum_{\xi\in\mathbb{Z}^n}e^{i2\pi x\xi+2\pi it\phi(t,x,\xi)}a(t,x,\xi)(\mathscr{F}_{{\mathbb{T}^n}}f)(\xi),\,0<t_0\leq t<\infty.
\end{equation} Let us assume that $\phi:[t_0,\infty)\times \mathbb{T}^n\times \mathbb{R}^n\rightarrow\mathbb{R}$ is a   real-valued phase function,  homogeneous of order $1$ in $\xi\neq 0,$ and satisfies
\begin{equation}
|\det(I+t\partial_x\partial_\xi\phi(t,x,\xi))|\geq C_0>0,\,\,|\partial_x^\beta\partial_\xi^\alpha \phi(t,x,\xi)|\leq C_{\alpha,\beta}t^{-|\beta|},\,\,t\geq t_0>0,
\end{equation} for all $x\in \mathbb{T}^n,$ $\xi\neq 0,$ and for  $1\leq |\beta|,|\alpha|\leq 2n+2.$ Let us assume that $a:[t_0,\infty)\times \mathbb{T}^n\times \mathbb{Z}^n\rightarrow\mathbb{C}$ is supported in $t|\xi|\geq C$ for some constant $C>0$ and that
\begin{equation}
|\partial_x^\beta\Delta_\xi^\alpha a(t,x,\xi)|\leq C_{\alpha,\beta}t^{-|\beta|},\,\,t\geq t_0>0,\,\,|\alpha|,|\beta|\leq 2n+2.
\end{equation} Then, the family $A_{t},\,0<t_0\leq t<\infty,$ is uniformly bounded on $L^2(\mathbb{T}^n).$ Moreover
\begin{equation}
\Vert A_t f\Vert_{L^2(\mathbb{T}^n)}\leq C\sup_{|\alpha|,|\beta|\leq 2n+2}C_{\alpha,\beta}\cdot  \Vert  f\Vert_{L^2(\mathbb{T}^n)}.
\end{equation}
\end{theorem}
Since the phase functions $\phi$ are usually considered homogeneous of order $1$ in $\xi\neq 0,$ i.e., $\phi(x,\lambda \xi)=\lambda\phi(x,\xi)$ we will work with phase functions defined on $\mathbb{T}^n\times \mathbb{R}^n$ instead of phase functions on $\mathbb{T}^n\times \mathbb{Z}^n.$ 

This paper is organized as follows. In  Section \ref{Preliminaries} we present some topics about the toroidal Fourier analysis  and the periodic analysis of pseudo-differential operators due to Ruzhansky and Turunen. In Sections \ref{Sec3} and \ref{Sec4} we prove our main results. Finally, in Section \ref{Sec5} we prove dispersive estimates for FSOs.

\section{Preliminaries}\label{Preliminaries}

In this section we present  some preliminaries on the analysis of periodic  pseudo-differential operators and FSOs. The main object in this {\textit{periodic analysis}} is the $n$-dimensional torus where the functions are defined.

\begin{definition}[The torus] The torus is the quotient space defined as follows: 
\begin{equation*}
\mathbb{T}^n=(\mathbb{R}/\mathbb{Z})^n = \mathbb{R}^n / \mathbb{Z}^n ,
\end{equation*}
where $\mathbb{Z}^n $ denotes the additive group of integral coordinates (the addition being, of course, the one derived from the vector structure of $\mathbb{R}^n).$ 
\end{definition}
Now, we recall the notion of  functions on the torus, identified with  $1$-periodic functions.
\begin{definition} [The 1-periodic functions] A function $f:\mathbb{R}^n \rightarrow \mathbb{C}$ is 1-periodic if 
\begin{equation*}
f(x+k)=f(x) \qquad \text{for every } x \in \mathbb{R}^n \text{ and } k \in \mathbb{Z}^n.
\end{equation*}
There is a clear one-to-one correspondence between these
functions on $\mathbb{R}^n$ and the functions on $\mathbb{T}^n$.
\end{definition}
The space of 1-periodic $m$-times continuously differentiable functions is denoted by $C^m(\mathbb{T}^n)$. We denote by 
$C^{\infty}(\mathbb{T}^n)=\bigcap\limits_{m \in \mathbb{Z}^n} C^m(\mathbb{T}^n)$ 
the space of the test functions. 
\begin{remark}
The convergence on $C^{\infty}(\mathbb{T}^n)$ is defined by: 
\begin{equation*}
u_j \rightarrow u,
\end{equation*}
if and only if 
\begin{equation*}
\partial^{\alpha}u_j\rightarrow \partial^{\alpha}u \text{ uniformly for all } \alpha \in \mathbb{N}_0^{n}.
\end{equation*}
\end{remark}
\begin{definition}[Schwartz space $\mathcal{S}(\mathbb{Z}^n)$]
A function $q:\mathbb{Z}^n\rightarrow \mathbb{C}$ belongs to $  \mathcal{S}(\mathbb{Z}^n)$ if for any $ M>0, $ there exists a constant $C_{q, M}>0$ such that 
\begin{equation*}
\lvert q (\xi) \rvert \leq C_{q, M} \langle \xi \rangle ^{-M},\,\,\langle\xi\rangle=(1+|\xi|^2)^\frac{1}{2},
\end{equation*}
holds for all $\xi \in  \mathbb{Z}^n$.
\end{definition}
\begin{remark}
The topology on $\mathcal{S}(\mathbb{Z}^n)$ is given by the seminorms $p_k$, where $k \in \mathbb{N}_0$ and \begin{equation*}
q_k(\phi)= \sup\limits_{\xi \in \mathbb{Z}^n} \langle \xi \rangle ^{k} |q(\xi)|.
\end{equation*}
\end{remark}
The Fourier transform furnishes the spectrum of the derivative  operator. So, since FSOs are motivated by applications to PDEs we need the Fourier transform.
\begin{definition}[Toroidal Fourier transformation $\mathscr{F}_{\mathbb{T}^n}$] 
The toroidal Fourier transformation is defined by: 
\begin{eqnarray*}
\mathscr{F}_{\mathbb{T}^n}: C^{\infty}(\mathbb{T}^n)  \rightarrow  \mathcal{S}(\mathbb{Z}^n) ,\,\,\,
 f  \mapsto  \hat{f}, 
\end{eqnarray*}
where 
\begin{equation*}
(\mathscr{F}_{\mathbb{T}^n}f)(\xi):=\hat{f}(\xi)= \int_{\mathbb{T}^n} e^{-i2\pi x \xi} f(x) dx .
\end{equation*}
Then, $\mathscr{F}_{\mathbb{T}^n}$ is a bijection and its inverse is given by: 
\begin{equation*}
f(x)= \sum_{\xi \in \mathbb{Z}^n} e^{i2\pi x \xi} \hat{f}(\xi)
\end{equation*}
\end{definition}
Now, we recall the definition of periodic Lebesgue spaces.
\begin{definition}[Spaces $L^p(\mathbb{T}^n)$]
For $1 \leq p < \infty$, let $L^p(\mathbb{T}^n)$ be the space of all $u\in L^1(\mathbb{T}^n)$ such that
\begin{equation*}
\lVert u \rVert_{L^p(\mathbb{T}^n)} = \left( \int_{\mathbb{T}^n} \lvert u (x)\rvert^p dx\right)^{ \frac{1}{p}} < \infty . 
\end{equation*}
For $p=\infty$, let $L^{\infty}(\mathbb{T}^n)$ be the space of all complex functions $u:\mathbb{T}^n\rightarrow \mathbb{C}$ such that
\begin{equation*}
\lVert u \rVert_{L^{\infty}(\mathbb{T}^n)} =  \text{esssup}_{x \in \mathbb{T}^n} \lvert u (x)\rvert < \infty. 
\end{equation*}
\end{definition}
The instrumental tool in the theory of FSOs is the notion of periodic symbol. Indeed, symbols allow us to classify operators by analytical and geometrical properties. For this, we need the following definition.
\begin{definition}[Space $C^{\infty}(\mathbb{Z}^n \times  \mathbb{T}^n)$] We say that a function $a \in  C^{\infty}(\mathbb{Z}^n \times  \mathbb{T}^n)$ if $a(., \xi)$ is smooth on $\mathbb{T}^n$ for all $\xi \in \mathbb{Z}^n $.
\end{definition}
 We denote by  $\delta_j \in \mathbb{N}_0^n,$ $j=1,\cdots, n,$ the canonical basis of $\mathbb{R}^n,$ namely, 
\begin{center}
$
(\delta_j)_i=
\begin{cases}
1, \qquad  \text{if } i=j \\ 
0, \qquad  \text{if } i\neq j.
\end{cases}
$ 
\end{center} Now, we recall the notion of discrete derivatives (difference operators).
\begin{definition}\label{diferenceoperator}
Let $\sigma : \mathbb{Z}^n \rightarrow \mathbb{C}$ and $1 \leq i,j \leq n$. We define the forward and backward partial difference operators $\Delta_{\xi_j},$ $\xi\in \mathbb{Z}^n,$ $j=1,\cdots,n,$ and $\overline{\Delta_{\xi_j}}$ respectively, by 
\begin{equation*}
\Delta_{\xi_j}\sigma(\xi)= \sigma(\xi+\delta_j)-\sigma(\xi),
\end{equation*}
\begin{equation*}
\overline{\Delta}_{\xi_j}\sigma(\xi)= \sigma(\xi)-\sigma(\xi-\delta_j).
\end{equation*}
Moreover,  for $\alpha \in \mathbb{N}_0^n,$ $\xi\in\mathbb{N}^n,$ define 
\begin{equation*}
\Delta_{\xi}^{\alpha}=\Delta_{\xi_1}^{\alpha_1} \ldots \Delta_{\xi_n}^{\alpha_n}.
\end{equation*}
\begin{equation*}
\overline{\Delta}_{\xi}^{\alpha}=\overline{\Delta}_{\xi_1}^{\alpha_1} \ldots \overline{\Delta}_{\xi_n}^{\alpha_n}.
\end{equation*}
Now, we consider symbols in $C^\infty(\mathbb{T}^n \times  \mathbb{Z}^n)$. These classes are motivated by the treatment of (periodic) elliptic and hypoelliptic  problems on  $\mathbb{T}^n$.
\end{definition}
\begin{definition}[Toroidal symbol class $S_{\rho, \delta}^m(\mathbb{T}^n \times  \mathbb{Z}^n)$]
Let $m \in \mathbb{R}$, $0 \leq \delta, \rho \leq 1 $. Then the toroidal symbol class $S_{\rho, \delta}^m(\mathbb{Z}^n \times  \mathbb{T}^n)$ consists of those functions $a(x, \xi) \in C^{\infty}(\mathbb{Z}^n \times  \mathbb{T}^n)$ which satisfy the toroidal symbol inequalities, that is, for $\alpha,\beta\in \mathbb{N}_0^n$ there exists $C_{\alpha,\beta}>0$ such that
\begin{equation*}
\lvert  \Delta_\xi ^\alpha \partial_x^\beta a (x, \xi)\rvert \leq C_{a\alpha \beta m}\langle \xi \rangle^{m- \rho \lvert \alpha \rvert + \delta \lvert \beta \rvert} ,
\end{equation*}
for every $x \in \mathbb{T}^n$ and for all $\xi \in \mathbb{Z}^n$.
\end{definition}
 The basic example of FSOs are pseudo-differential operators. We recall the following definition (see \cite{Ruz}).
\begin{definition}[Toroidal pseudo-differential operators]
If $a \in S_{\rho, \delta}^m(\mathbb{T}^n \times  \mathbb{Z}^n)$, we denote by $a(X,D)$ the corresponding toroidal pseudo-differential operator defined by 
\begin{equation} \label{pdo}
a(X,D)f(x)= \sum_{\xi \in \mathbb{Z}^n} e^{i2 \pi x \xi} a(x, \xi) \hat{f}(\xi).
\end{equation}
The series (\ref{pdo}) converges, e.g.,  if $f \in C^{\infty}(\mathbb{T}^n)$.  The set of operators  of the form (\ref{pdo}) with  $a \in S_{\rho, \delta}^m(\mathbb{T}^n \times  \mathbb{Z}^n)$ is denoted by $\textnormal{Op}(S_{\rho, \delta}^m(\mathbb{T}^n \times  \mathbb{Z}^n))$. If an operator $A$ satisfies $A \in \textnormal{Op}(S_{\rho, \delta}^m(\mathbb{T}^n \times  \mathbb{Z}^n))$ we denote its toroidal symbol by $\sigma_A=\sigma_A(x,\xi), \quad x \in \mathbb{T}^n, \quad \xi \in \mathbb{Z}^n$.
\end{definition}
\begin{proposition}[Difference formula for symbols] Let $\sigma_A\in C^k(\mathbb{T}^n\times \mathbb{Z}^n).$ For every $\alpha\in \mathbb{N}_0^n$ and $\beta\in \mathbb{N}_0^n,$ $|\beta|\leq k$ we have the identity
\begin{equation*}
\Delta_\xi ^\alpha \partial_x^\beta \sigma_A(x, \xi)=\sum_{\gamma \leq \alpha}(-1)^{\alpha - \gamma} \left( \begin{array} {c} \alpha  \\ \gamma \end{array}\right) \partial_x^\beta \sigma_A(x, \xi+ \gamma),
\end{equation*} for every $(x,\xi)\in \mathbb{T}^n\times \mathbb{Z}^n.$
\end{proposition}

\begin{definition} [Toroidal amplitudes]
The class $\mathcal{A}_{\rho, \delta}^m (\mathbb{T}^n \times \mathbb{T}^n \times \mathbb{Z}^n)$ of toroidal amplitudes consists of the functions $a(x, y, \xi)$ which are smooth in $x$ and $y$ for all $\xi \in \mathbb{Z}^n$ and which satisfy: 
\begin{equation*}
\lvert  \Delta_\xi ^\alpha \partial_x^\beta \partial_y^\gamma a(x,y, \xi)\rvert \leq C_{a\alpha \beta \gamma m}\langle \xi \rangle^{m- \rho \lvert \alpha \rvert + \delta \lvert \beta +\gamma \rvert} .
\end{equation*}
We may define: 
\begin{equation*}
a(X,Y,D)f(x)= \sum_{\xi \in \mathbb{Z}^n} \int_{\mathbb{T}^n} e^{i2 \pi (x-y) \xi} a(x,y, \xi) f(y)dy. \qquad \text{for } f \in C^{\infty}(\mathbb{T}^n),
\end{equation*}
\end{definition}
We present, in a more general form, the definition of FSOs.
\begin{definition}[Amplitude Fourier series operators]
Amplitude Fourier series operators (AFSOs) are operators of the form: 
\begin{equation}
Tu(x)=\sum_{\xi \in \mathbb{Z}^n} \int_{\mathbb{T}^n} e^{2 \pi i (\phi(x, \xi) -y\xi)} a(x,y, \xi) u(y)dy,
\end{equation}
where $a \in C^{\infty}(\mathbb{T}^n \times \mathbb{T}^n \times \mathbb{Z}^n)$ is a toroidal amplitude and $\phi$ is a real-valued \textit{phase function} such that $x \rightarrow e^{2 \pi i \phi(x, \xi)}$ is 1-periodic for all $\xi \in \mathbb{Z}^n$. In this paper we also use the term amplitude periodic Fourier integral operators for AFSOs.
\end{definition}
In our further analysis  we will use the close relation between toroidal symbols (in the sense of Ruzhansky and Turunen) and periodic H\"ormander classes. We introduce such classes as follow.
\begin{definition}[Periodic symbol class $S_{\rho, \delta}^m(\mathbb{T}^n \times  \mathbb{R}^n)$]
Symbols in $S^m_{\rho,\delta}(\mathbb{T}^n\times \mathbb{R}^n)$ are symbols in $S^m_{\rho,\delta}(\mathbb{R}^n\times \mathbb{R}^n)$ (see \cite{Hor1, Ruz}) of order $m$ which are 1-periodic in $x.$
If $a(x,\xi)\in S^{m}_{\rho,\delta}(\mathbb{T}^n\times \mathbb{R}^n),$ the corresponding pseudo-differential operator is defined by
\begin{equation}\label{hh}
a(x,D_x)u(x)=\int_{\mathbb{T}^n}\int_{\mathbb{R}^n}e^{i2\pi\langle x-y,\xi \rangle}a(x,\xi)u(y)d\xi dy,\,\, u\in C^{\infty}(\mathbb{T}^n).
\end{equation}
\end{definition}

In our further analysis we will use  Corollary 4.5.7 of \cite{Ruz} which we present as Corollary \ref{IClases} below. 

\begin{corollary}\label{IClases}\ Let $0\leq \delta \leq 1,$ $0\leq \rho<1.$ Let $a:\mathbb{T}^n\times \mathbb{R}^n\rightarrow \mathbb{C}$ satisfying
\begin{equation}
\lvert  \partial_\xi ^\alpha \partial_x^\beta a (x, \xi)\rvert \leq C^{(1)}_{a\alpha \beta m}\langle \xi \rangle^{m- \rho \lvert \alpha \rvert + \delta \lvert \beta \rvert} ,
\end{equation}
for $|\alpha|\leq N_{1}$ and $|\beta|\leq N_{2}.$ Then the restriction $\tilde{a}=a|_{\mathbb{T}^n\times \mathbb{Z}^n}$ satisfies the estimate
\begin{equation}
\lvert  \Delta_\xi ^\alpha \partial_x^\beta \tilde{a} (x, \xi)\rvert \leq C_{a\alpha \beta m}C^{(1)}_{a\alpha \beta m}\langle \xi \rangle^{m- \rho \lvert \alpha \rvert + \delta \lvert \beta \rvert} ,
\end{equation}
for $|\alpha|\leq N_{1}$ and $|\beta|\leq N_{2}.$  The converse holds true, i.e, if a symbol $\tilde{a}(x,\xi)$ on $\mathbb{T}^n\times \mathbb{Z}^n$ satisfies  $(\rho,\delta)$-inequalities of the form 
\begin{equation}
\lvert  \Delta_\xi ^\alpha \partial_x^\beta \tilde{a} (x, \xi)\rvert \leq C^{(2)}_{a\alpha \beta m}\langle \xi \rangle^{m- \rho \lvert \alpha \rvert + \delta \lvert \beta \rvert} ,
\end{equation}
then $\tilde{a}(x,\xi)$ is  the restriction of a symbol $a(x,\xi)$ on $\mathbb{T}^n\times \mathbb{R}^n$ satisfying estimates of the type
\begin{equation}
\lvert  \partial_\xi ^\alpha \partial_x^\beta a (x, \xi)\rvert \leq C_{a\alpha \beta m}C^{(2)}_{a\alpha \beta m}\langle \xi \rangle^{m- \rho \lvert \alpha \rvert + \delta \lvert \beta \rvert} .
\end{equation}
\end{corollary}
Let us denote $\Psi^{m}_{\rho,\delta,N_1,N_2}(\mathbb{T}^n\times \mathbb{Z}^n)$ to the set of operators associated with symbols satisfying $(\rho,\delta)$-estimates for all $|\alpha|\leq N_{1}$ and $|\beta|\leq N_2,$ and $\Psi^{m}_{\rho,\delta,N_1,N_2}(\mathbb{T}^n\times \mathbb{R}^n)$ defined similarly. Then we have the following equivalence (see Theorem 2.14 of \cite{Profe}):
\begin{proposition}
$\label{eqcc}($Equality of operators classes$).$ For $0\leq \delta \leq 1,$ $0<\rho\leq 1$ we have $\Psi^{m}_{\rho,\delta,N_1,N_2}(\mathbb{T}^n\times \mathbb{Z}^n)=\Psi^{m}_{\rho,\delta,N_1,N_2}(\mathbb{T}^n\times \mathbb{R}^n).$
\end{proposition}
In the next section we generalize the following classical result (see Theorem 3.8 of Stein and Weiss \cite{SteWei}).

\begin{proposition}\label{stein}
Suppose $1\leq p\leq \infty$ and $T_\sigma$ be a Fourier multiplier on $ \mathbb{R}^n$ with symbol $\sigma(\xi)$. If  $\sigma(\xi)$  is continuous at each point of $\mathbb{Z}^n$ then the periodic operator defined by
\begin{equation}
\sigma(D)f(x)=\sum_{\xi \in \mathbb{Z}^n}e^{i2\pi x\xi}\sigma(\xi)\widehat{u}(\xi),
\end{equation}
is a bounded operator from $L^p(\mathbb{T}^n)$ into $L^p(\mathbb{T}^n).$
\end{proposition} It is important to mention that vector valued pseudo-differential on the torus can be found in \cite{Bie} and references therein. The reference \cite{Duvan4} presents a complete classification for periodic pseudo-differential operators on $L^p$-spaces.

\section{From the boundedness of FIOs to the boundedness of SFOs}\label{Sec3}
In this section we address two problems. The first one, is the interaction between  the boundedness of FIOs and its periodic counterpart. More precisely we will investiagte how the boundedness of FIOs implies the boundedness of  FSOs. The second problem that we analyse is how to extend the boundedness of FSOs  with symbols depending only on the frequency variable, i.e FSOs of the type 
\begin{equation}
A_{\phi,a}f(x):=\sum_{\xi\in\mathbb{Z}^n}e^{2\pi i\phi(x,\xi)}a(\xi)(\mathscr{F}_{{\mathbb{T}^n}}f)(\xi),
\end{equation} to general FSOs where the symbols depend of both variables $x$ and $\xi.$ If we fix a phase function $\phi$ defined on $\mathbb{T}^n\times \mathbb{Z}^n,$ we denote by $\mathfrak{S}_{\phi,p}(\mathbb{T}^n)$  the set  of  symbols  $a:\mathbb{Z}^n\rightarrow \mathbb{C}$ with periodic Fourier integral operator $A_{\phi,a}$ admitting a bounded extension  $A_{\phi,a}:L^p(\mathbb{T}^n)\rightarrow L^p(\mathbb{T}^n).$ We define on $\mathfrak{S}_{\phi,p}(\mathbb{T}^n)$ the natural norm
\begin{equation}
\Vert a\Vert_{\mathfrak{S}_{\phi,p}(\mathbb{T}^n)}:=\Vert A_{\phi,a} \Vert_{\mathscr{B}(L^p(\mathbb{T}^n))},\,\,1\leq p\leq \infty.
\end{equation}
To be precise,  we will define the class of periodic Fourier integral operators that we will investigate.
\begin{definition}
A continuous linear operator $A:C^\infty(\mathbb{T}^n)\rightarrow\mathscr{D}'(\mathbb{T}^n)$ is a periodic Fourier integral operator, if there exist a real-valued phase function $\phi:\mathbb{T}^n\times\mathbb{Z}^n\rightarrow \mathbb{R} ,$  homogeneous of order $1$ in $\xi\neq 0,$ and a symbol $a\in S^m_{1,0}(\mathbb{T}^n\times \mathbb{Z}^n)$ such that \begin{equation}
Af(x)=A_{\phi,a}f(x):=\sum_{\xi\in\mathbb{Z}^n}e^{2\pi i\phi(x,\xi)}a(\xi)(\mathscr{F}_{{\mathbb{T}^n}}f)(\xi),\,\,f\in C^\infty(\mathbb{T}^n).
\end{equation} 
\end{definition}

Our main theorem in this section is the following generalisation of one classical result for Fourier multipliers (see Theorem \ref{stein} or Theorem 3.8 of Stein and Weiss \cite[Chapter VII]{SteWei}). The corresponding statement for FSOs  is the following. 

\begin{theorem}\label{teorema principal11}
Let $1< p<\infty.$ Let us assume that $\phi$ is a real valued continuous function defined on $\mathbb{T}^n\times\mathbb{R}^n.$ If  $a:\mathbb{T}^n\times\mathbb{R}^n\rightarrow \mathbb{C}$ is a continuous bounded function and the Fourier integral operator
\begin{equation}
T_{\phi,a}f(x):=\int_{\mathbb{R}^n}e^{2\pi i\phi(x,\xi)}a(x,\xi)(\mathscr{F}_{\mathbb{R}^n}f)(\xi)d\xi
\end{equation} extends to a bounded operator  $T_{\phi,a}:L^p(\mathbb{R}^n)\rightarrow L^p(\mathbb{R}^n),$ then the periodic Fourier integral operator
\begin{equation}\label{cp'}
A_{\phi,a}f(x):=\sum_{\xi\in\mathbb{Z}^n}e^{2\pi i\phi(x,\xi)}a(x,\xi)(\mathscr{F}_{{\mathbb{T}^n}}f)(\xi)
\end{equation} also extends to a bounded operator $A_{\phi,a}:L^p(\mathbb{T}^n)\rightarrow L^p(\mathbb{T}^n).$ Moreover, for  some  $C_p>0,$ the estimate $\Vert A_{\phi,a} \Vert_{\mathscr{B}(L^p(\mathbb{T}^n))}\leq C_p\Vert T_{\phi,a}\Vert_{\mathscr{B}(L^p(\mathbb{R}^n))}$ holds true.
\end{theorem}
We begin with the proof of this result by considering the following technical  lemmas. The next result is presented as  Lemma 3.9 in \cite{SteWei}.
\begin{lemma}\label{lemma0} Suppose $f$ is a continuous periodic function on $\mathbb{R}^n.$ If $\{\epsilon_m\}$ is a sequence of positive  real numbers, then
\begin{equation}
\lim_{m\rightarrow\infty}\epsilon_m^{\frac{n}{2}}\int_{\mathbb{R}^n}e^{-\epsilon_m|x|^2}f(x)dx=\int_{\mathbb{T}^n}f(x)dx
\end{equation} provided that $\epsilon_{m}\rightarrow 0.$
\end{lemma}
For our further analysis will be useful the following extension of Lemma \ref{lemma0}.
\begin{lemma}\label{lemma1} Suppose $f$ is a continuous periodic function on $\mathbb{R}^n$ and let $\{g_{m}\}$ be a sequence of uniformly bounded continuous periodic functions on $\mathbb{R}^n.$ If $g_{m}$ converges pointwise to a function $g$ defined on $\mathbb{R}^n$ and $\{\epsilon_m\}$ is a sequence of positive real numbers, then
\begin{equation}
\lim_{m\rightarrow\infty}\epsilon_m^{\frac{n}{2}}\int_{\mathbb{R}^n}e^{-\epsilon_m|x|^2}f(x)g_{m}(x)dx=\int_{\mathbb{T}^n}f(x)g(x)dx
\end{equation} provided that $\epsilon_{m}\rightarrow 0.$
\end{lemma}
\begin{proof}
By using Lemma \ref{lemma0}, we obtain for every $m\in\mathbb{N}$
$$I_{1,m}:=\lim_{s\rightarrow\infty}  \epsilon_s^{\frac{n}{2}}\int_{\mathbb{R}^n}e^{-\epsilon_s|x|^2}f(x)g_{m}(x)dx=\int_{\mathbb{T}^n}f(x)g_m(x)dx.                  $$ Now, taking into account that the sequence $\{g_m\}$ is uniformly bounded, an application of the dominated convergence  theorem gives
$$I_{2,s}:=\lim_{m\rightarrow\infty}  \epsilon_s^{\frac{n}{2}}\int_{\mathbb{R}^n}e^{-\epsilon_s|x|^2}f(x)g_{m}(x)dx=\epsilon_s^{\frac{n}{2}}\int_{\mathbb{T}^n}e^{-\epsilon_s|x|^2}f(x)g(x)dx.                  $$ 
So, the limit $\lim_{m,s\rightarrow\infty}  \epsilon_s^{\frac{n}{2}}\int_{\mathbb{R}^n}e^{-\epsilon_s|x|^2}f(x)g_{m}(x)dx                 $ of the double sequence \begin{equation}\label{doubleseq}
    \left\{\epsilon_s^{\frac{n}{2}}\int_{\mathbb{R}^n}e^{-\epsilon_s|x|^2}f(x)g_{m}(x)dx\right\}_{m,s}
\end{equation}  exists and can be computed from iterated limits in the following way:
\begin{align*}
     \lim_{m,s\rightarrow\infty}  \epsilon_s^{\frac{n}{2}}\int_{\mathbb{R}^n}e^{-\epsilon_s|x|^2}f(x)g_{m}(x)dx &= \lim_{m\rightarrow\infty} \lim_{s\rightarrow\infty}  \epsilon_s^{\frac{n}{2}}\int_{\mathbb{R}^n}e^{-\epsilon_s|x|^2}f(x)g_{m}(x)dx\\ &=\lim_{m\rightarrow\infty} I_{1,m}=\int_{\mathbb{T}^n}f(x)g(x)dx,
\end{align*}
where in the last line we have use the dominated convergence theorem. Because
$$  \left\{\epsilon_m^{\frac{n}{2}}\int_{\mathbb{R}^n}e^{-\epsilon_m|x|^2}f(x)g_{m}(x)dx\right\}_{m\in \mathbb{N}}$$ is a sub-sequence of \eqref{doubleseq},
 we obtain
$$\lim_{m\rightarrow\infty}  \epsilon_m^{\frac{n}{2}}\int_{\mathbb{R}^n}e^{-\varepsilon_m|x|^2}f(x)g_{m}(x)dx=\int_{\mathbb{T}^n}f(x)g(x)dx,                 $$ as claimed. 
\end{proof}

\begin{proof}[Proof of Theorem \ref{teorema principal11}] First, let us assume that $P$ and $Q$ are trigonometric polynomials. For every $\delta>0$ let us denote by $w_\delta(x)=e^{-\delta|x|^2}.$ So, if $\varepsilon,\alpha,\beta>0$ and $\alpha+\beta=1$ let us note that
\begin{equation}\label{eq1}
\lim_{\varepsilon \rightarrow 0}\varepsilon^{\frac{n}{2}}\int_{\mathbb{R}^n}[ T(Pw_{\varepsilon\alpha})(x)]\overline{Q(x)}w_{\varepsilon \beta}(x)dx=(\pi/\beta)^{n/2}\int_{\mathbb{T}^n}(A P)(x)\overline{Q(x)}dx.
\end{equation}
By linearity we only need to prove \eqref{eq1} when $P(x)=e^{i2\pi m x}$ and $Q(x)=e^{i2\pi k x}$ for $k$ and $m$ in $\mathbb{Z}^n.$ The right hand side of \eqref{eq1} can be computed as follows,
\begin{align*}
\int_{\mathbb{T}^n}(A P)(x)\overline{Q(x)}dx &=\int_{\mathbb{T}^n}\left(\sum_\xi e^{2\pi i\phi(x,\xi)}{a(x,\xi)\widehat{P}(\xi)} \right)\overline{Q(x)}dx\\
&=\int_{\mathbb{T}^n}\left(\sum_\xi e^{2\pi i\phi(x,\xi)}{a(x,\xi)\delta_{m,\xi}}
\right)\overline{Q(x)}dx\\
&=\int_{\mathbb{T}^n} e^{2\pi i\phi(x,m)}a(x,m)\overline{Q(x)}dx=\int_{\mathbb{T}^n} e^{2\pi i\phi(x,m)-i2\pi kx }a(x,m)dx.
\end{align*}
Now, we compute the left hand side of \eqref{eq1}. Taking under consideration that the euclidean Fourier transform of $Pw_{\alpha\varepsilon}$ is given by
\begin{equation}
[\mathscr{F}_{\mathbb{R}^n}(Pw_{\alpha\varepsilon})](\xi)=(\alpha\varepsilon)^{-\frac{n}{2}}e^{-|\xi-m|^2/\alpha\varepsilon},
\end{equation}
by the Fubini theorem we have 
\begin{align*}
&\int_{\mathbb{R}^n}    [T(Pw_{\varepsilon\alpha})(x)]    \overline{Q(x)}  w_{\varepsilon \beta}(x)dx\\ &=\int_{\mathbb{R}^n}\int_{\mathbb{R}^n}e^{2\pi i\phi( x,\xi)}a(x,\xi) (\alpha\varepsilon)^{-\frac{n}{2}}e^{-|\xi-m|^2/\alpha\varepsilon} \overline{Q(x)}w_{\varepsilon \beta}(x)d\xi dx\\
&=\int_{\mathbb{R}^n}\left(\int_{\mathbb{R}^n}e^{2\pi i\phi( x,\xi)-i2\pi kx} e^{-\pi\varepsilon\beta|x|^2} a(x,\xi) dx \right)(\alpha\varepsilon)^{-\frac{n}{2}}e^{-|\xi-m|^2/\alpha\varepsilon} d\xi\\
&=\int_{\mathbb{R}^n}\left(\int_{\mathbb{R}^n}e^{2\pi i\phi( x,(\alpha\varepsilon)^{\frac{1}{2}}\eta+m)-i2\pi kx} a( x,(\alpha\varepsilon)^{\frac{1}{2}} \eta+m) e^{-\pi\varepsilon\beta|x|^2}dx \right) e^{-|\eta|^2} d\eta.\\
\end{align*}
So, we have
\begin{align*}
&\lim_{\varepsilon\rightarrow 0}\varepsilon^{n/2} \int_{\mathbb{R}^n}   [ T(Pw_{\varepsilon\alpha})](x)   \overline{Q(x)}  w_{\varepsilon \beta}(x)dx\\
&=\lim_{\varepsilon\rightarrow 0}\beta^{-\frac{n}{2}}(\beta\varepsilon)^{n/2}\int_{\mathbb{R}^n}   [ T(Pw_{\varepsilon\alpha})](x)    \overline{Q(x)}  w_{\varepsilon \beta}(x)dx\\
&=\lim_{\varepsilon\rightarrow 0}\beta^{-\frac{n}{2}}(\beta\varepsilon)^{n/2} \int_{\mathbb{R}^n}\int_{\mathbb{R}^n}e^{2\pi i\phi( x,(\alpha\varepsilon)^{\frac{1}{2}}\eta+m)-i2\pi kx}a( x,(\alpha\varepsilon)^{\frac{1}{2}} \eta+m) e^{-\pi\varepsilon\beta|x|^2}dx \,\cdot\, e^{-|\eta|^2} d\eta.
\end{align*}
By Lemma \ref{lemma1}, we have
\begin{align*}
\lim_{\varepsilon\rightarrow 0} (\beta\varepsilon)^{n/2} &\int_{\mathbb{R}^n}e^{2\pi i\phi( x,(\alpha/\beta)^{\frac{1}{2}}(\beta\varepsilon)^{\frac{1}{2}}\eta+m)}e^{-i2\pi kx}a( x,(\alpha\varepsilon)^{\frac{1}{2}} \eta+m) e^{-\pi\varepsilon\beta|x|^2}dx\\
&=\int_{\mathbb{T}^n}e^{2\pi i\phi( x,m)-i2\pi kx}a(x,m)dx.
\end{align*}
Taking into account that $\int_{\mathbb{R}^n}e^{-|\eta|^2}d\eta=\pi^{n/2},$ and that $a$ is a  continuous bounded function, by the dominated convergence theorem we have
\begin{align*}
\lim_{\varepsilon\rightarrow 0}\varepsilon^{n/2} \int_{\mathbb{R}^n}    [T(Pw_{\varepsilon\alpha})  ](x) \overline{Q(x)}  w_{\varepsilon \beta}(x)dx =(\pi/\beta)^{n/2}\int_{\mathbb{T}^n}e^{2\pi i\phi( x,m)-i2\pi kx}a(x,m)dx.
\end{align*}
If we assume that $T$ is a bounded linear operator on $L^p(\mathbb{R}^n),$ then the restriction of $A$ to trigonometric polynomials is a bounded operators on $L^p(\mathbb{T}^n).$ In fact, if $\alpha=\frac{1}{p}$ and $\beta=\frac{1}{p'}$ we obtain
\begin{align*}
&\Vert AP\Vert_{L^p(\mathbb{T}^n)} =\sup_{\Vert Q \Vert_{L^{p'}(\mathbb{T}^n)}=1}\left|\int_{\mathbb{T}^n}(A P)(x)\overline{Q(x)}dx \right|\\
&=\sup_{\Vert Q \Vert_{L^{p'}(\mathbb{T}^n)}=1}\lim_{\varepsilon\rightarrow 0} \varepsilon^{n/2} (\frac{1}{\pi p'})^{n/2}\left| \int_{\mathbb{R}^n}    [T(Pw_{\varepsilon\alpha})](x)   \overline{Q(x)}  w_{\varepsilon \beta}(x)dx  \right|\\
&\leq\sup_{\Vert Q \Vert_{L^{p'}(\mathbb{T}^n)}=1}\lim_{\varepsilon\rightarrow 0} \varepsilon^{n/2}  (\frac{1}{\pi p'})^{n/2}\Vert T \Vert_{\mathscr{B}(L^p)}\Vert Pw_{\varepsilon/p}\Vert_{L^p(\mathbb{R}^n)}\Vert Qw_{\varepsilon/p'}\Vert_{L^{p'}(\mathbb{T}^n)}\\
&\leq\sup_{\Vert Q \Vert_{L^{p'}(\mathbb{T}^n)}=1} \Vert T \Vert_{\mathscr{B}(L^p)}\lim_{\varepsilon\rightarrow 0}  (\frac{1}{\pi p'})^{n/2} \left(    \varepsilon^{n/2}\int_{\mathbb{R}^n}|P(x)|^pe^{-\pi\varepsilon|x|^2}dx\right)^{\frac{1}{p}} \\ 
&\hspace{8cm} \times\left(  \varepsilon^{n/2}  \int_{\mathbb{R}^n}|Q(x)|^{p'}e^{-\pi\varepsilon|x|^2}dx   \right)^{\frac{1}{p'}}\\
&\leq\sup_{\Vert Q \Vert_{L^{p'}(\mathbb{T}^n)}=1} \Vert T \Vert_{\mathscr{B}(L^p)} (\frac{1}{\pi p'})^{n/2} \left(  \int_{\mathbb{T}^n}|P(x)|^pdx\right)^{\frac{1}{p}}  \left(   \int_{\mathbb{T}^n}|Q(x)|^{p'}dx   \right)^{\frac{1}{p'}}\\
&= \Vert T \Vert_{\mathscr{B}(L^p)} (\frac{1}{\pi p'})^{n/2} \Vert P\Vert_{L^p(\mathbb{T}^n)}.
\end{align*}
Since the restriction of $A$ to trigonometric polynomials is a bounded operator on $L^p(\mathbb{T}^n)$ this restriction admits a unique bounded extension on $L^p(\mathbb{T}^n).$ The proof is complete.
\end{proof}
\begin{remark}
From the proof of \label{teorema principal1}, let us observe that the constant $C_p$  in \eqref{cp} can be estimated by $C_p\leq \Vert T \Vert_{\mathscr{B}(L^p)} (\frac{1}{\pi p'})^{n/2}. $
\end{remark}
\begin{remark}
As it can be observed from the proof of Theorem  \ref{teorema principal11}, the assumptions on the boundedness of the second argument in the phase function and the symbol $a$ can be imposed only on $\xi\in \mathbb{Z}^n.$ \end{remark}

\begin{theorem}\label{extension}
Let us choose a phase function $1$-periodic in $x$  and let $1<p<\infty.$ If $a:\mathbb{T}^n\times \mathbb{Z}^n\rightarrow \mathbb{C}$ is a function satisfying
\begin{equation}
\sup_{z\in\mathbb{T}^n}\Vert \partial_z^\alpha a(z,\cdot)\Vert_{\mathfrak{S}_{\phi,p}(\mathbb{T}^n)}<\infty,\,\,\,|\alpha|\leq \left[\frac{n}{p}\right]+1,
\end{equation} then the Fourier integral operator 
\begin{equation}
Af(x):=A_{\phi,a}f(x)=\sum_{\xi\in\mathbb{Z}^n}e^{2\pi i\phi(x,\xi)}a(x,\xi)(\mathscr{F}_{{\mathbb{T}^n}}f)(\xi),
\end{equation}  extends to a bounded operator  $A:L^p(\mathbb{T}^n)\rightarrow L^p(\mathbb{T}^n).$
\end{theorem}

\begin{proof}
 For every $z\in\mathbb{T}^n,$ we define the operator family given by
\begin{equation}
A_{z}f(x):=\sum_{\xi\in\mathbb{Z}^n}e^{2\pi i\phi(x,\xi)}a(z,\xi)(\mathscr{F}_{{\mathbb{T}^n}}f)(\xi).
\end{equation} Taking into account the identity $A_xf(x)=Af(x),$  by the Sobolev embedding Theorem we have
\begin{equation}
    \sup_{z\in\mathbb{T}^n}|A_zf(x)|\lesssim \sum_{|\beta|\leq [\frac{n}{p}]+1}\Vert\partial_z^\beta A_zf(x)\Vert_{L^p(\mathbb{T}^n_z)}=\sum_{|\beta|\leq [\frac{n}{p}]+1}\left(\int_{\mathbb{T}^n} \vert\partial_z^\beta A_zf(x)\vert^p \, dz\right)^\frac{1}{p}.
\end{equation}
Consequently, we have
\begin{align*}
\Vert Af \Vert^p_{L^p(\mathbb{T}^n)}&=\int_{\mathbb{T}^n}|A_xf(x)|^pdx\leq \int_{\mathbb{T}^n}\sup_{z\in\mathbb{T}^n}|A_zf(x)|^pdx\\
&\lesssim \sum_{|\beta|\leq [\frac{n}{p}]+1} \int_{\mathbb{T}^n}\int_{\mathbb{T}^n} \vert\partial_z^\beta A_zf(x)\vert^p \, dzdx\\&=\sum_{|\beta|\leq [\frac{n}{p}]+1} \int_{\mathbb{T}^n}\int_{\mathbb{T}^n} \vert\partial_z^\beta A_zf(x)\vert^p \, dxdz \\
&= \sum_{|\beta|\leq [\frac{n}{p}]+1}\int_{\mathbb{T}^n}\Vert\partial_z^\beta A_zf\Vert^p_{L^p(\mathbb{T}^n)}dz\\
&\leq \sum_{|\beta|\leq [\frac{n}{p}]+1}\sup_{z\in\mathbb{T}^n}\Vert\partial_z^\beta A_z\Vert_{\mathscr{B}(L^p(\mathbb{T}^n))} \Vert f\Vert^p_{L^p(\mathbb{T}^n)}  \\
&= \sum_{|\beta|\leq [\frac{n}{p}]+1} \sup_{z\in\mathbb{T}^n}\Vert \partial_z^\beta a(z,\cdot)\Vert^p_{\mathfrak{S}_{\phi,p}(\mathbb{T}^n)} \Vert f\Vert^p_{L^p(\mathbb{T}^n)} . 
\end{align*} So we have $ \Vert Af \Vert_{L^p(\mathbb{T}^n)}\leq C\Vert f \Vert_{L^p(\mathbb{T}^n)}  $ where
$$ C^p= \sum_{|\beta|\leq [\frac{n}{p}]+1} \sup_{z\in\mathbb{T}^n}\Vert \partial_z^\beta a(z,\cdot)\Vert^p_{\mathfrak{S}_{\phi,p}(\mathbb{T}^n)} <\infty. $$ The proof is complete.
\end{proof}

\section{Boundedness of periodic Fourier integral operators}\label{Sec4}
In this section we present sufficient conditions for the $L^p$-boundedness of FSOs. Instead of the conditions presented in the previous section now we consider symbol criteria for the boundedness of these operators. Our main theorem in this section is the next Theorem \ref{T3}.

\begin{theorem}\label{T3} 
Let us assume that $\phi:\mathbb{T}^n\times \mathbb{R}^n\rightarrow\mathbb{R}$ is a   real-valued phase function positively homogeneous of order $1$ in $\xi\neq 0$. Let us assume that $\partial_{x}^{\gamma'}\partial_{\xi}^\gamma\phi\in S^0_{0,0}(\mathbb{T}^n\times (\mathbb{R}^n\setminus \{0\}))$ when $|\gamma|=|\gamma'|=1,$ that
\begin{equation}
    |\textnormal{det}(\partial_{y}\partial_{\xi}\phi(y,\xi))|\geq C>0,\,\,\,|\partial^\alpha_{y}\phi(y,\xi)|\leq C_{\alpha}|\xi|,\,\,\xi\neq 0,
\end{equation}
\begin{equation}
    \langle \nabla_{\xi}\phi(y,\xi)\rangle\asymp  1,\,\,\,\langle \nabla_{y}\phi(y,\xi)\rangle\asymp \langle \xi\rangle, 
\end{equation} 
and the symbol inequalities
\begin{equation}
|\partial_x^\beta\Delta_\xi^\alpha a(x,\xi)|\leq C_{\alpha,\beta}\langle \xi\rangle^{\mu-|\alpha|},\,\,\mu\leq\mu_p:=-(n-1)\left|\frac{1}{p}-\frac{1}{2}\right|,\,\,|\beta|\leq \left[\frac{n}{p}\right]+1,
\end{equation} hold true. Then, the periodic Fourier integral operator in \eqref{cp} extends to a bounded linear operator $A:=A_{\phi,a}:L^p(\mathbb{T}^n)\rightarrow L^p(\mathbb{T}^n)$  for all  $1<p<\infty.$
\end{theorem}
Theorem \ref{T3} is a version on the torus $\mathbb{T}^n$ (with conditions of limited regularity in the spatial variables) of one proved by M. Ruzhansky and S. Coriasco \cite{CoRu1,CoRu2} for Fourier integral operators on $\mathbb{R}^n$. The corresponding assertion is the following.
\begin{theorem}\label{CorRuzhansky}
Let us assume that $\phi:\mathbb{R}^n\times \mathbb{R}^n\rightarrow\mathbb{R}$ is a   real-valued phase function positively homogeneous of order $1$ in $\xi\neq 0$. Let us assume that 
\begin{equation}
    |\textnormal{det}(\partial_{y}\partial_{\xi}\phi(y,\xi))|\geq C>0,\,\,\,|\partial^\alpha_{y}\phi(y,\xi)|\leq C_{\alpha}\langle y\rangle^{1-|\alpha|}|\xi|,\,\,\xi\neq 0,
\end{equation}
\begin{equation}
    \langle \nabla_{\xi}\phi(y,\xi)\rangle\asymp  \langle y\rangle,\,\,\,\langle \nabla_{y}\phi(y,\xi)\rangle\asymp \langle \xi\rangle, 
\end{equation} that
$\partial_{x}^{\gamma'}\partial_{\xi}^\gamma\phi\in S^0_{0,0}(\mathbb{R}^n\times( \mathbb{R}^n\setminus\{0\} ))$ when $|\gamma|=|\gamma'|=1,$ and 
\begin{equation}\label{CRcostantes}
|\partial_x^\beta\partial_\xi^\alpha a(x,y,\xi)|\leq C_{\alpha,\beta}\langle x\rangle^{m_1-|\beta|}\langle y\rangle^{m_2-|\gamma|}\langle \xi\rangle^{\mu-|\alpha|},\,\,\mu\leq\mu_p:=-(n-1)\left|\frac{1}{p}-\frac{1}{2}\right|,
\end{equation} holds true for all $\alpha,\beta,\gamma\in\mathbb{N}_0^n$ provided that $m_1+m_2=m\leq m_p=-n\left|\frac{1}{p}-\frac{1}{2}\right|$. Then, the  Fourier integral operator $T:=T_{\phi,a}:L^p(\mathbb{R}^n)\rightarrow L^p(\mathbb{R}^n)$ extends to a bounded operator for all  $1<p<\infty.$
\end{theorem}
\begin{remark}
A look to the proof of Theorem \ref{CorRuzhansky} allows us to conclude that the operator norm $\Vert T \Vert_{\mathscr{B}(L^p)}$ satisfies estimates of the type
\begin{equation}\label{estimativodenorma}
\Vert T \Vert_{\mathscr{B}(L^p)}\leq C \sup_{|\alpha|,|\beta|\leq \ell}C_{\alpha,\beta}
\end{equation} where the constants $C_{\alpha,\beta}$ where defined in \eqref{CRcostantes}, and $\ell$ is a positive, large enough integer.
\end{remark}

\begin{proof}[Proof of Theorem \ref{T3}] 
Now, we can use the machinery developed in the previous section. If $a$ is a symbol satisfying $|\Delta_\xi^\alpha a(\xi)|\leq C_{\alpha,\beta}\langle \xi\rangle^{\mu-\alpha},$  $\mu\leq\mu_p:=-(n-1)|\frac{1}{p}-\frac{1}{2}|,$ by Corollary \ref{IClases} there exists $\tilde{a}:\mathbb{R}^n\rightarrow\mathbb{C}$ such that $|\partial_\xi^\alpha a(\xi)|\leq C_{\alpha,\beta}\langle \xi\rangle^{\mu-\alpha},\,\,\mu\leq\mu_p:=-(n-1)\left|\frac{1}{p}-\frac{1}{2}\right|$ and $a=\tilde{a}|_{\mathbb{Z}^n}.$ Taking this into account the Fourier integral operator $T=T_{\phi,a}$ associated with $\phi$ and the  symbol $\tilde{a}$ is bounded on $L^p(\mathbb{R}^n)$ (in fact $T$ satisfies the hypothesis in Theorem \ref{CorRuzhansky}). So, by Theorem \ref{teorema principal1} $A$ extends to a bounded operator on $L^p(\mathbb{T}^n)$ and from \eqref{estimativodenorma} we have
\begin{equation}
\Vert A\Vert_{\mathscr{B}(L^p)}\leq C_p \Vert T\Vert_{\mathscr{B}(L^p)}\leq C_p C \sup_{|\alpha|\leq \ell}{C_{\alpha}}
\end{equation} where the constants $C_\alpha$ are defined by the condition
\begin{equation}
|\partial_\xi^\alpha \tilde{a}(\xi)|\leq C_{\alpha}\langle \xi\rangle^{\mu-|\alpha|},\,\,\mu\leq\mu_p:=-(n-1)\left|\frac{1}{p}-\frac{1}{2}\right|.
\end{equation} 
So, we finish the proof for this case. Now, if the symbol $a(x,\xi)$ depends on $x,$  let us define for every $z\in\mathbb{T}^n,$  the operator  given by
\begin{equation}
A_{z}f(x):=\sum_{\xi\in\mathbb{Z}^n}e^{2\pi i\phi(x,\xi)}a(z,\xi)(\mathscr{F}_{{\mathbb{T}^n}}f)(\xi).
\end{equation} Since
\begin{equation}
|\partial_\xi^\alpha (\partial_z^\beta{a})(z,\xi)|\leq C_{\alpha}\langle \xi\rangle^{\mu-|\alpha|},\,\,\mu\leq\mu_p:=-(n-1)\left|\frac{1}{p}-\frac{1}{2}\right|,\,|\beta|\leq \left[\frac{n}{p}\right]+1,
\end{equation} we have
\begin{equation}
\sup_{z\in\mathbb{T}^n}\Vert \partial_z^\beta{a} \Vert_{\mathfrak{S}_{\phi,p}}\lesssim \sup_{|\alpha|,|\beta|\leq \max\{\ell,[n/p]+1\}}C_{\alpha,\beta}<\infty.
\end{equation}
To conclude, we now only have to apply Theorem
 \ref{extension}.
\end{proof}

\section{Dispersive estimates for periodic Fourier integral operators}\label{Sec5}
In this section we prove some dispersive $L^2$-estimates for a parametrized family of periodic Fourier integral operators of the form,
\begin{equation}\label{cptt'}
A_{t}f(x):=\sum_{\xi\in\mathbb{Z}^n}e^{2\pi ix\xi+2\pi it\phi(t,x,\xi)}a(t,x,\xi)(\mathscr{F}_{{\mathbb{T}^n}}f)(\xi),\,0<t_0\leq t<\infty.
\end{equation}
Our starting point is the following result due to M. Ruzhansky and J. Wirth \cite{RWirth}.

\begin{theorem}[Ruzhansky-Wirth]\label{RuzWirthTheo} Let us consider the parametrized family of FIOs
\begin{equation}\label{cptt''''}
T_{t}f(x):=\int_{\mathbb{R}^n}e^{i2\pi x\xi+2\pi it\phi(t,x,\xi)}a(t,x,\xi)(\mathscr{F}_{{\mathbb{R}^n}}f)(\xi)d\xi,\,0<t_0\leq t<\infty.
\end{equation} Let us assume that $\phi(t,x,\xi)$ is real-valued, $1$-homogeneous in $\xi$ and satisfies
\begin{equation}
|\det(I+t\partial_x\partial_\xi\phi(t,x,\xi))|\geq C_0>0,\,\,|\partial_x^\beta\partial_\xi^\alpha \phi(t,x,\xi)|\leq C_{\alpha,\beta}t^{-|\beta|},\,\,t\geq t_0>0,
\end{equation} for all $x\in \mathbb{R}^n,$ $\xi\neq 0,$ and for  $1\leq |\alpha|,|\beta|\leq 2n+2.$ Let us assume that $a(t,x,\xi)$ is supported in $t|\xi|\geq C$ for some constant $C>0$ and that
\begin{equation}
|\partial_x^\beta\partial_\xi^\alpha a(t,x,\xi)|\leq C_{\alpha,\beta}t^{-|\beta|},\,\,t\geq t_0>0,\,\,|\alpha|,|\beta|\leq 2n+2.
\end{equation} Then, the family $T_{t}f(x),\,0<t_0\leq t<\infty,$ is uniformly bounded on $L^2(\mathbb{R}^n).$ Moreover
\begin{equation}
\Vert T_t f\Vert_{L^2(\mathbb{R}^n)}\leq\left( C\sup_{|\alpha|,|\beta|\leq 2n+2}C_{\alpha,\beta}\right)\cdot  \Vert  f\Vert_{L^2(\mathbb{R}^n)}.
\end{equation}
\end{theorem}

To conclude, we prove the following result for  parametrized families of periodic Fourier integral operators.

\begin{theorem}Let us consider the parametrized family of periodic Fourier integral operators 
\begin{equation}\label{cptt'''''X}
A_{t}f(x):=\sum_{\xi\in\mathbb{Z}^n}e^{i2\pi x\xi+2\pi it\phi(t,x,\xi)}a(t,x,\xi)(\mathscr{F}_{{\mathbb{T}^n}}f)(\xi),\,0<t_0\leq t<\infty.
\end{equation} Let us assume that $\phi:[t_0,\infty)\times \mathbb{T}^n\times \mathbb{R}^n\rightarrow\mathbb{R}$ is a   real-valued phase function,  homogeneous of order $1$ in $\xi\neq 0,$ and satisfies
\begin{equation}
|\det(I+t\partial_x\partial_\xi\phi(t,x,\xi))|\geq C_0>0,\,\,|\partial_x^\beta\partial_\xi^\alpha \phi(t,x,\xi)|\leq C_{\alpha,\beta}t^{-|\beta|},\,\,t\geq t_0>0,
\end{equation} for all $x\in \mathbb{T}^n,$ $\xi\neq 0,$ and for  $1\leq |\beta|,|\alpha|\leq 2n+2.$ Let us assume that $a:[t_0,\infty)\times \mathbb{T}^n\times \mathbb{Z}^n\rightarrow\mathbb{C}$ is supported in $t|\xi|\geq C$ for some constant $C>0$ and that
\begin{equation}
|\partial_x^\beta\Delta_\xi^\alpha a(t,x,\xi)|\leq C_{\alpha,\beta}t^{-|\beta|},\,\,t\geq t_0>0,\,\,|\alpha|,|\beta|\leq 2n+2.
\end{equation} Then, the family $A_{t},\,0<t_0\leq t<\infty,$ is uniformly bounded on $L^2(\mathbb{T}^n).$ Moreover
\begin{equation}
\Vert A_t f\Vert_{L^2(\mathbb{T}^n)}\leq C\sup_{|\alpha|,|\beta|\leq 2n+2}C_{\alpha,\beta}\cdot  \Vert  f\Vert_{L^2(\mathbb{T}^n)}.
\end{equation}
\end{theorem}
\begin{proof}
If $a(t,x,\xi)$ is a periodic symbol satisfying \begin{equation}
|\partial_x^\alpha\Delta_\xi^\beta a(t,\xi)|\leq C_{\alpha,\beta}t^{-|\alpha|},\,\,t\geq t_0>0,\,\,|\alpha|,|\beta|\leq 2n+2,
\end{equation}  by Corollary \ref{IClases} there exists $\tilde{a}:[t_0,\infty)\times\mathbb{T}^n\times \mathbb{R}^n\rightarrow\mathbb{C}$ such that \begin{equation}
|\partial_x^\alpha\partial_\xi^\beta a(t,x,\xi)|\leq C_{\alpha,\beta}t^{-|\alpha|},\,\,t\geq t_0>0,\,\,|\alpha|,|\beta|\leq 2n+2,
\end{equation}  and $a(t,\cdot,\cdot)=\tilde{a}(t,\cdot,\cdot)|_{\mathbb{T}^n\times\mathbb{Z}^n}.$ Taking this into account, the family of Fourier integral operators $T_t=T_{\phi,a(t,\cdot)}$ associated with the phase function $\phi$ and  the  symbol $\tilde{a}(t,x,\cdot),$ is uniformly bounded on $L^2(\mathbb{R}^n)$ (in fact, the family $T_t$ satisfies the hypothesis in Theorem \ref{RuzWirthTheo}). So, by Theorem \ref{teorema principal1}, the parametrized family $A_t$ extends to a uniformly bounded one  on $L^2(\mathbb{T}^n)$ and from \eqref{estimativodenorma} we have
\begin{equation}
\Vert A_t\Vert_{\mathscr{B}(L^2)}\leq C_2 \Vert T_t\Vert_{\mathscr{B}(L^2)}\leq C_2 C \sup_{|\alpha|\leq 2n+2}{C_{\alpha}}
\end{equation} So, we finish the proof.
\end{proof}

\vspace{.0cm}
\noindent{\textbf{Acknowledgements.}} The authors are indebted to an anonymous referee  for very asserted suggestions which have improved the analysis and presentation of this manuscript. We thank Jan Rozendaal for comments.

\bibliographystyle{amsplain}

\begin{thebibliography}{99}
\bibitem{ag} M. S.Agranovich, Spectral properties of elliptic pseudodifferential operators on a closed curve, Funct. Anal. Appl, \textbf{13}, pp. 279-281 (1971).
 
\bibitem{AF} K. Asada, D. Fujiwara, On some oscillatory integral transformations in $L^2(\mathbb{R}^n),$ 
J. Math. (N.S.),Japan, \textbf{4}(2),pp. 299--361, (1978).

\bibitem{va} R. Ashino, M. Nagase and R. Vaillancourt, Pseudodifferential operators on $L^p(\mathbb{R}^n)$ spaces, Cubo, \textbf{6}, N 3. pp. 91-129, (2004).

\bibitem{Bie} B. B.Mart\'inez, R. Denk, J. H.Monz\'on and T. Nau, Generation of semigroups for vector-valued pseudodifferential operators on the torus, J. Fourier Anal. Appl, \textbf{22}(4), pp. 823--853, (2016).

\bibitem{BealsM} M. Beals, $L^p$ Boundedness of Fourier integral operators, Mem. Amer. Math. Soc, \textbf{38}(264), pp. viii+57, (1982).


\bibitem{BeltranHickmanSogge2018} D. Beltran, J. Hickman and C. Sogge, Sharp local smoothing estimates for Fourier integral operators, arXiv:1812.11616.

\bibitem{SENOUS2} S. Bekkara,  B. Messirdi and A. Senoussaoui, A class of generalised integral operators, Elect. J. Diff. Eq, Vol, No. 88, pp. 1--7, (2009).


\bibitem{Duvan2} D. Cardona, Estimativos $L^2$ para una clase de operadores pseudodiferenciales definidos en el toro, Rev. Integr. Temas Mat, \textbf{31}(2), pp. 147--152, (2013).

\bibitem{Duvan3}  D. Cardona, Weak type $(1, 1)$ bounds for a class of periodic pseudodifferential operators, J. Pseudo-Differ. Oper. Appl, \textbf {5}(4),  pp. 507--515, (2014).

\bibitem{Duvan4} D. Cardona, On the boundedness of periodic pseudo-differential operators, Monat. Math, \textbf{185}(2), pp. 189--206, (2017).  

\bibitem{Duvan5} D. Cardona, Pseudo-differential operators on $\mathbb{Z}^n$ with applications to discrete fractional integral operators, Bull. Iran. Math. Soc. To appear. doi: 10.1007/s41980-018-00195-y, arXiv:1803.00231.  

\bibitem{KumarCardona} D. Cardona,   V. Kumar,  Multilinear analysis for discrete and periodic pseudo-differential operators in Lp spaces,  Rev. Integr. temas Mat. Vol 36,  (2) (2018), 151-164.

\bibitem{KumarCardona2} D. Cardona,   V. Kumar, Lp-boundedness and Lp-nuclearity of multilinear pseudo-differential operators on Zn and the torus Tn, J. Fourier Anal. Appl. to appear, arXiv:1809.08380.

\bibitem{CoRu1} S. Coriasco, M. Ruzhansky, On the boundedness of Fourier integral operators on $L^p(\mathbb{R}^n),$ C. R. Math. Acad. Sci. Paris, \textbf{348}(15--16), pp. 847--851, (2010).

\bibitem{CoRu2} S. Coriasco, M. Ruzhansky, Global $L^p$
continuity of Fourier integral operators, Trans.
Amer. Math. Soc, \textbf{366}(5), pp. 2575--2596, (2014).



\bibitem{Profe} J. Delgado, $L^p$ bounds for pseudo-differential operators on the torus, Operators Theory, advances and applications. \textbf{231}, pp.  103--116, (2012).

\bibitem{DR4} J. Delgado, M. Ruzhansky, $L^{p}$ -bounds for  pseudodifferential operators on compact Lie groups, Journal of the institute of mathematics of Jussieu, pp. 1--29, (2017) doi:10.1017/S1474748017000123.

\bibitem{Duo} J. Duoandikoetxea, Fourier Analysis, Amer. Math. Soc, (2001).

\bibitem{DuiHor} J. J.Duistermaat, H\"ormander, Fourier integral operators. II, Acta Math, \textbf{128}(3-4), pp. 183--269, (1972).

\bibitem{Dui1} J. J.Duistermaat, Fourier integral operators, volume 130 of Progress in mathematics, Birkh\"auser Boston, Inc. Boston, MA, (1996).

\bibitem{Elo} O. Elong, A. Senoussaoui, On the $L^p$ boundedness of certain class of semiclassical Fourier operators,  Matematicki Vesnik, \textbf{70}(3), pp. 189--203, 2018

\bibitem{Eskin}  G.I. Eskin, Degenerate elliptic pseudodifferential equations of principal  type, Mat. Sb. (N.S.), \textbf{82}(124), pp. 585--628, (1970).


 \bibitem{Fuji} A. Fujiwara, Construction of the fundamental solution for the Schr\"odinger equations, Proc. Japan Acad. Ser. A Math. Sc, \textbf{55}(1), pp. 10–14, (1979).
 



\bibitem  {Har} C. Harrat, A. Senoussaoui, On a class of h-Fourier integral operators, Demostratio Mathematica, Vol. XLVII. N3, pp. 594--606, (2014).  



\bibitem{Hor71}  L. H\"ormander, Fourier integral operators. I,  Acta Math, \textbf{127}(1-2), pp. 79--183, (1971).

\bibitem{Hor1} L. H\"ormander,  Pseudodifferential operators and Hypo-elliptic equations.  Proc. symposium on singular integrals, Amer. Math. Soc, \textbf{10}, pp. 138--183, (1967).

\bibitem{Hor2} L. H\"ormander, The Analysis of the linear partial differential operators, Vol. III. IV, Springer-Verlag, (1985).
 \bibitem{Kumano-go}  H. Kumano-go, A calculus of Fourier integral operators on $\mathbb{R}^n$ and the fundamental solution
for an operator of hyperbolic type, Comm. Partial differential equations, \textbf{1}(1), pp. 1--44, (1976).

\bibitem{JKohnLNirenberg} J. Kohn, L. Nirenberg, On the algebra of pseudodifferential operators, Comm. Pure. Appl. Math, \textbf{18}, pp. 269-- 305, (1965).

\bibitem{Mc} W.M. Mclean, Local and global description of periodic pseudodifferential operators, Math. Nachr, \textbf{150},  pp. 151--161,  (1991).

\bibitem{SENOUS1} B. Messirdi, A. Senoussaoui, $L^2$ boundedness and $L^2$ compactness of a class of Fourier integral operators, Electronical journal of differential equations, Vol. 2006, No. 26, pp. 1--12, (2006).
\bibitem{Miyachi} A. Miyachi, On some estimates for the wave equation in $L^
p$ and $H^p$, J. Fac. Sci. Univ. Tokyo Sect. IA Math, \textbf{27}(2), pp. 331–-354, (1998).

\bibitem{s1} S. Molahajloo, A characterization of compact pseudodifferential operators on $\mathbb{S}^1$, Oper. Theory Adv. Appl, Birkh\"user/Springer Basel AG, Basel, \textbf{213}, pp. 25--29, (2011).

\bibitem{s2} S. Molahajloo, M.W. Wong, Pseudodifferential operators on $\mathbb{S}^1$, New developments in pseudodifferential operators, Eds. L. Rodino and M.W. Wong, pp. 297--306, (2008).

\bibitem{m} S. Molahajloo, M.W. Wong, Ellipticity, Fredholmness and spectral invariance of pseudodifferential operators on $\mathbb{S}^1,$ J. Pseudo-Differ. Oper. Appl, \textbf{1}, pp. 183--205, (2010).



\bibitem{Peral}  J. C.Peral,  $L^p$-estimates for the wave equation, J. Funct. Anal, \textbf{36}(1), pp. 114--145, (1980).

\bibitem{RuzSugi01}  M. Ruzhansky, M. Sugimoto, Global $L^2$-boundedness theorems for a class of Fourier integral operators. Comm. Partial Differential Equations, \textbf{31}(4--6), pp. 547--569, (2006).

\bibitem{RuzSugi02}  M. Ruzhansky,  M. Sugimoto, A smoothing property of Schr\"odinger equations in the critical case, Math. Ann, \textbf{335}(3), pp. 645--673, (2006).

\bibitem{RuzSugi03}  M. Ruzhansky,  M. Sugimoto, Weighted Sobolev $L^2$
estimates for a class of Fourier integral operators, Math. Nachr, \textbf{284}(13), pp. 1715--1738, (2011).

\bibitem{RuzSugi} M. Ruzhansky,  M. Sugimoto, Global regularity properties for a class of Fourier integral operators, arxiv. 

\bibitem{Ruz} M. Ruzhansky, V. Turunen, Pseudodifferential operators and symmetries: Background Analysis and Advanced Topics, Birkha\"user-Verlag, Basel, (2010).

\bibitem{Ruz-2} M. Ruzhansky, V. Turunen, Quantization of pseudoDifferential operators on the torus, J.  Fourier. Annal. Appl, Vol. 16, pp. 943--982, Birkh\"auser Verlag, Basel, (2010).

\bibitem{RWirth} M. Ruzhansky, J. Wirth, Dispersive type estimates for Fourier integrals and applications to hyperbolic systems, Conference Publications, 2011, (Special), pp. 1263--1270. doi: 10.3934/proc.2011.2011.1263, (2011)


\bibitem{M. Ruzhansky} M. Ruzhansky, Regularity theory of Fourier integral operators with complex phases and singularities of affine fibrations, Volume 131 of CWI Tract, Stichting Mathematisch Centrum, Centrum voor Wiskunde en Informatica, Amsterdam, (2001).



\bibitem{SSS91} A. Seeger, C. D.Sogge and E. M.Stein, Regularity properties of Fourier integral operators, Ann. of Math. (2), \textbf{134}(2), pp. 231--251, '(1991).

\bibitem{SteWei} E.M. Stein, G. Weiss, Introduction to Fourier analysis on Euclidean spaces, Princeton University Press, Princeton, N.J, (1971).

\bibitem{Eli} E.M. Stein, Harmonic analysis: real-variable methods, orthogonality, and oscillatory integrals, Princeton University Press, (1993).

\bibitem{Tao} T. Tao, The weak-type $(1,1)$ of Fourier integral operators of order $-(n - 1)/2,$  J. Aust. Math. Soc, \textbf{76}(1), pp. 1--21, (2004).
\bibitem{tur} V. Turunen, G. Vainikko, On symbol analysis of periodic pseudodifferential operators, Z. Anal. Anwendungen, \textbf{17}, pp. 9--22 (1998).
\bibitem{LW} L. Wang, Pseudo-differential operators with rough coefficients, 
Thesis (Ph.D.)–McMaster University (Canada), ProQuest LLC, Ann Arbor, MI, 1997. 66 pp. ISBN: 978--0612--30120--7. 


\end{thebibliography}

\end{document}